\documentclass[11pt,a4paper,reqno]{amsart}
\usepackage[margin=25mm]{geometry}
\usepackage[T1]{fontenc}
\usepackage{lmodern}
\usepackage{amsmath,amssymb,amsthm,mathtools,booktabs,array}
\usepackage{xcolor,fancyhdr,needspace}
\usepackage[colorlinks=true,linkcolor=blue!50!black,citecolor=blue!50!black,urlcolor=blue!50!black]{hyperref}
\hypersetup{pdftitle={Integral Volume Criteria for Semilinear and Quasilinear Liouville Theorems on Complete Manifolds},pdfauthor={Yongheng Qin and Yuhua Sun},pdfsubject={Integral volume criteria for elliptic and parabolic Liouville theorems},pdfkeywords={Liouville theorems, elliptic inequalities, p-Laplacian, nonlinear parabolic equations, Cauchy problem, Fujita exponent, volume growth, test-function method}}
\newtheorem{theorem}{Theorem}[section]
\newtheorem{lemma}[theorem]{Lemma}
\newtheorem{proposition}[theorem]{Proposition}
\newtheorem{corollary}[theorem]{Corollary}
\theoremstyle{remark}\newtheorem{remark}[theorem]{Remark}
\newcommand{\dd}{\,\mathrm d}

\newcommand{\supp}{\operatorname{supp}}
\newcommand{\R}{\mathbb R}
\newcommand{\Q}{\mathcal Q}
\newcommand{\eps}{\varepsilon}
\newcommand{\ind}{\mathbf1}
\numberwithin{equation}{section}
\title[Integral volume Liouville criteria]{Integral Volume Criteria for Semilinear and Quasilinear Liouville Theorems on Complete Manifolds}
\author{Yongheng Qin}
\address{School of Mathematical Sciences and LPMC, Nankai University,
  Tianjin 300071, P. R. China}
\email{1120220027@mail.nankai.edu.cn}

\author{Yuhua Sun}
\address{School of Mathematical Sciences and LPMC, Nankai University,
  Tianjin 300071, P. R. China}
\email{sunyuhua@nankai.edu.cn}

\thanks{\noindent
  Y. Sun was funded by the National Natural Science Foundation of
  China (Grant No. 12371206) and the Fundamental Research Funds for the Central Universities, No. 050-63263078.}

\makeatletter
\@ifundefined{subjclassname@2020}{%
  \@namedef{subjclassname@2020}{\textup{2020} Mathematics Subject Classification}%
}{}
\makeatother
\subjclass[2020]{Primary 35B53; Secondary 35J92, 35K92, 58J35}
\keywords{Liouville theorems, elliptic inequalities, $p$-Laplacian, nonlinear parabolic equations, Cauchy problem, Fujita exponent, volume growth, test-function method}
\date{}
\begin{document}
\begin{abstract}
Let $M$ be a connected complete noncompact Riemannian manifold without boundary, and let $V_o(R)$ denote its ball volume.
We prove Liouville theorems without any curvature assumption.
For $p>1$ and $q>p-1$, divergence of
$\int^\infty R^{pq/(p-1)-1}V_o(R)^{-(q-p+1)/(p-1)}\dd R$
forces every nonnegative Sobolev weak supersolution of $-\Delta_pu\ge u^q$ to vanish.
For $q>\max\{1,p-1\}$, divergence of
$\int^\infty R^{p-1}V_o(R)^{-(q-p+1)}\dd R$
forces every nonnegative global distributional energy supersolution of $u_t-\Delta_pu\ge u^q$ to vanish.
Consequently, the corresponding Cauchy problem has no global solution in this class with a nonzero nonnegative local $L^1$ initial trace.
\end{abstract}
\maketitle

\section{Introduction}
The existence of positive supersolutions of nonlinear differential inequalities reflects both the strength of the reaction and the geometry at infinity.
For a fixed exponent $q>1$, two basic problems are
\begin{equation}\label{eq:models}
 -\Delta u\ge u^q\quad\text{on }M,
 \qquad
 u_t-\Delta u\ge u^q\quad\text{on }M\times(0,\infty),
 \qquad u\ge0.
\end{equation}
Here $\Delta=\operatorname{div}\nabla$ is the Laplace--Beltrami operator, and the measure is Riemannian volume.
The corresponding Liouville questions ask when the zero function is the only supersolution in a specified local solution class.
The parabolic question also provides nonexistence results for the global Cauchy problem.
In Euclidean space, the critical powers are different: the elliptic supersolution threshold is $n/(n-2)$ when $n>2$, whereas the parabolic threshold is $1+2/n$.
On a general manifold, volume growth and potential theory provide geometric replacements for the Euclidean dimension.

\subsection{Elliptic background}
The test-function method derives nonexistence from integral inequalities in the weak formulation.
Its systematic development for nonlinear differential equations and inequalities, including the nonlinear capacity viewpoint, can be found in Mitidieri and Pohozaev~\cite{MP01}.
An advantage of using the first-order weak formulation is that localization can be expressed in terms of the gradient of a distance cutoff, without differentiating the distance twice.
This is relevant on general complete manifolds, where completeness alone supplies no useful global bound for the Laplacian of the distance function.

Grigor'yan and Kondratiev~\cite{GK10} developed capacity and volume tests on manifolds.
For the elliptic inequality in \eqref{eq:models}, their volume criterion includes
\[
 V_o(R)\le C R^{2q/(q-1)}(\log R)^b,
 \qquad b<\frac1{q-1},
\]
in their positive weak solution class.
Grigor'yan and Sun~\cite[Theorem~1.1]{GS14} subsequently reached the sharp single-logarithmic endpoint
\begin{equation}\label{eq:knownElog}
 V_o(R)\le C R^{2q/(q-1)}(\log R)^{1/(q-1)}
 \quad\Longrightarrow\quad u=0
\end{equation}
for nonnegative Sobolev weak supersolutions on a connected geodesically complete manifold.
Their proof already uses a negative-power test and a pure cutoff test, with an auxiliary exponent tied to the cutoff scale and a preliminary global $L^q$ estimate. Their model examples show the sharpness of the polynomial and single-logarithmic powers.
Thus neither the use of two tests nor the single-logarithmic endpoint is new to the present argument.

Zhang~\cite[Theorem~1.2]{Z15} obtained nonexistence under the condition
\begin{equation}\label{eq:Zhang}
 \liminf_{\tau\downarrow0}\tau^{q/(q-1)}
 \int_1^\infty\frac{V_o(R)}{R^{(3q-1)/(q-1)+\tau}}\dd R<\infty.
\end{equation}
This condition contains the volume itself and an auxiliary small parameter.
It is different from the reciprocal-volume divergence condition considered below; no general equivalence between the two conditions is used here.

A complementary approach proceeds through Green functions.
Grigor'yan and Sun~\cite{GS19} studied positive classical supersolutions in exterior domains under volume growth and two-sided Green-function estimates.
Xu~\cite{Xu19} obtained existence on a restricted geometric class under a uniform-in-center convergence condition for a reciprocal-volume integral.
In particular, such existence results must be read together with their standing geometric assumptions.
Grigor'yan, Sun and Verbitsky~\cite[Theorem~1.1 and Corollary~1.2]{GSV20} proved that, under volume doubling and a two-sided Li--Yau type Green-function estimate, a positive classical supersolution of $-\Delta u\ge u^q$ exists precisely when
\begin{equation}\label{eq:historicalE}
 \int^\infty\frac{R^{2q-1}}{V_o(R)^{q-1}}\dd R<\infty.
\end{equation}
They formulated in Conjecture~1 the assertion that divergence of this integral should imply nonexistence on every connected complete manifold, without those additional hypotheses.
The elliptic theorem proved below establishes that nonexistence implication in the nonnegative Sobolev class specified in the statement.
It does not extend the converse existence assertion to arbitrary manifolds.

\subsection{Parabolic background}
For the Euclidean equation $u_t-\Delta u=u^q$, Fujita~\cite{Fujita} established the fundamental distinction between nonexistence below $q_F=1+2/n$ and small-data global existence above it.
The critical power also belongs to the nonexistence range.
Hayakawa~\cite{Hayakawa1973} treated the critical cases in dimensions one and two, and the work of Sugitani~\cite{Sugitani1975} and Kobayashi, Sirao and Tanaka~\cite{KST1977} includes the general-dimensional endpoint.
These results concern the global Cauchy problem, and their initial-data assumptions should not be confused with an elliptic supersolution condition.

On complete manifolds, Mastrolia, Monticelli and Punzo~\cite{MMP2017} developed a test-function theory for weak solutions of quasilinear parabolic inequalities with a potential, using weighted spacetime volume conditions.
For the ordinary Laplacian and reaction potential equal to one, their Corollary~1.6 and Remark~1.7(ii) in the arXiv version imply triviality under
\begin{equation}\label{eq:knownPlog}
 V_o(R)\le C R^{2/(q-1)}(\log R)^b,
 \qquad 0\le b\le\frac1{q-1}.
\end{equation}
The single-logarithmic parabolic endpoint therefore already belongs to the theory without a curvature hypothesis.
Related heat-kernel methods for semilinear inequalities and Fujita exponents on sub-Riemannian manifolds and unimodular Lie groups were developed by Ruzhansky and Yessirkegenov~\cite{RY2022}.
Their heat propagation assumptions form part of the relevant statements and are not consequences of geodesic completeness alone.

The parabolic reciprocal-volume integral
\begin{equation}\label{eq:historicalP}
 \int^\infty\frac{R}{V_o(R)^{q-1}}\dd R
\end{equation}
has also appeared directly in the literature.
Xu's preprint~\cite[Theorem~1.4]{Xu2018} gives divergence and small-data convergence statements under its standing conditions (G) and (H).
Condition (G) controls a radial logarithmic volume-density derivative and is used in the nonexistence proof to estimate the Laplacian of radial cutoffs; condition (H) is a heat-kernel upper bound used on the existence side.
A published chapter of Sun and Xu~\cite{SX2022} also treats the relation between this integral and the Cauchy problem.
The comparison with (G) and (H) made here refers specifically to the inspected 2018 preprint; no identification of the full hypotheses of different versions is assumed.

More recently, von Criegern, Grillo and Monticelli~\cite{CGM2025} extended the spacetime-volume test-function framework to broader nonlinear diffusions with potentials.
In a discrete setting, Gu, Hao, Huang and Sun~\cite{GHHS2026} proved an integral Fujita criterion for weighted graphs with a proper adapted path metric, without doubling or heat-kernel bounds in the nonexistence direction.
Their small-data existence result has an additional heat-kernel hypothesis.
The discrete theorem provides a useful comparison, while the spatial chain rules and Sobolev approximation required in the manifold setting are justified separately below.

\subsection{Nonlinear diffusion and the \texorpdfstring{$p$}{p}-Laplacian}
The nonlinear operator
\[
 \Delta_pu=\operatorname{div}(|\nabla u|^{p-2}\nabla u),\qquad p>1,
\]
is also commonly denoted by $\Delta_m$ with $m=p$.
The diffusion exponent must be distinguished from a polynomial volume degree.
The natural reaction range is $q>p-1$ for the elliptic argument and
$q>\max\{1,p-1\}$ for the parabolic argument.
The corresponding critical volume degrees are
\[
 D_E(p,q)=\frac{pq}{q-p+1},\qquad
 D_P(p,q)=\frac{p}{q-p+1}.
\]
On Euclidean space these produce the elliptic supersolution threshold
$n(p-1)/(n-p)$ when $n>p$, and the Fujita threshold $p-1+p/n$, respectively.
The latter gives a superlinear nonexistence interval only when it exceeds one.

The nonlinear setting and logarithmic volume corrections have established precedents.
Mastrolia, Monticelli and Punzo~\cite{MMP15} treated weighted elliptic $p$-Laplacian inequalities with a potential, and Sun~\cite{Sun16} studied a broader structural class of quasilinear operators.
For unit potential their results include the sufficient endpoint
\[
 V_o(R)\le C R^{pq/(q-p+1)}(\log R)^{(p-1)/(q-p+1)}.
\]
The parabolic work~\cite{MMP2017} and the more recent open-cylinder results~\cite{CGM2025} yield, for the ordinary $p$-Laplacian, sufficient conditions of the form
\[
 V_o(R)\le C R^{p/(q-p+1)}(\log R)^b,
 \qquad
 0\le b\le\min\left\{\frac1{q-1},\frac{p-1}{q-p+1}\right\}.
\]
These quoted bounds do not assert optimality of their logarithmic powers.
Our parabolic argument below uses a volume-dependent time length and proves a different integral condition directly, rather than inferring it from scaling or from those sufficient pointwise bounds.

On the existence side, Grillo, Meglioli and Punzo~\cite{GMP23} constructed global solutions for appropriate small data above the Euclidean Fujita threshold under Sobolev hypotheses; their further results with a Poincar\'e inequality have additional parameter restrictions.
These analytic assumptions are part of the existence theory and are not supplied by completeness or by convergence of a volume integral alone.
In a discrete setting, Gu, Hao, Huang and Sun~\cite{GHHS26p} obtained a reciprocal-volume criterion for the graph $p$-Laplacian.
That work is a direct precedent for the elliptic integral expression; it is not used to justify the continuum Sobolev or temporal tests in this paper.

\subsection{Linear-diffusion results and solution classes}
Throughout the paper, $(M,g)$ is smooth, connected, geodesically complete, noncompact, and without boundary, of dimension $n\ge1$.
In the linear-diffusion Sections~2--6, fix $o\in M$ and $q>1$, and write
\[
 \begin{gathered}
 r(x)=d(o,x),\qquad B_o(R)=\{r<R\},\qquad V_o(R)=\mu(B_o(R)),\\
 \gamma=q-1,\qquad \beta=\gamma^{-1},\qquad \kappa=\frac{2q}{\gamma}.
 \end{gathered}
\]
Completeness makes each closed ball compact, so radial Lipschitz cutoffs have compact support.
Only this fact, positivity and monotonicity of ball volumes, and $|\nabla r|\le1$ almost everywhere are needed from the geometry.
\emph{No curvature assumption, either pointwise or integral, is imposed on the manifold in these Liouville criteria.}

\begin{theorem}[Elliptic integral criterion]\label{thm:E}
Suppose that, for some $o\in M$ and $R_0>0$,
\begin{equation}\label{eq:Ecrit}
 \int_{R_0}^\infty\frac{R^{2q-1}}{V_o(R)^{q-1}}\dd R=\infty.
\end{equation}
Then $u=0$ almost everywhere for every function satisfying
\begin{equation}\label{eq:Eweak}
 \begin{gathered}
 u\in W^{1,2}_{\mathrm{loc}}(M)\cap L^q_{\mathrm{loc}}(M),\qquad u\ge0,\\
 \int_M\langle\nabla u,\nabla\zeta\rangle\dd\mu
 \ge\int_Mu^q\zeta\dd\mu
 \quad\text{for every }0\le\zeta\in C_c^\infty(M)
 \end{gathered}
\end{equation}
\end{theorem}

\begin{theorem}[Parabolic integral criterion]\label{thm:P}
Suppose that, for some $o\in M$ and $R_0>0$,
\begin{equation}\label{eq:Pcrit}
 \int_{R_0}^\infty\frac{R}{V_o(R)^{q-1}}\dd R=\infty.
\end{equation}
Then every nonnegative $u\in C^{2,1}(M\times(0,\infty))$ satisfying
\begin{equation}\label{eq:Psup}
 u_t-\Delta u\ge u^q\quad\text{on }M\times(0,\infty)
\end{equation}
vanishes identically.
In particular, if $0\le u_0\in C(M)$ and $u_0\not\equiv0$, the Cauchy problem
\begin{equation}\label{eq:Cauchy}
 u_t-\Delta u=u^q,\qquad u(\cdot,0)=u_0
\end{equation}
has no nonnegative global solution in
$C(M\times[0,\infty))\cap C^{2,1}(M\times(0,\infty))$.
\end{theorem}

Here $C^{2,1}$ means that the spatial derivatives through order two and the first temporal derivative exist and are continuous.
The parabolic proof begins at any positive time, so the first conclusion needs neither an initial trace nor global boundedness.
The elliptic and parabolic solution classes are deliberately stated separately.
This initial statement is classical; Section~\ref{sec:mP} supplies the temporal regularization and proves the corresponding criterion in the specified distributional energy class, also when $p=2$.
Its conclusion is global nonexistence; a norm-specific finite-time blow-up statement would additionally require a local well-posedness and continuation framework.

\subsection{The nonlinear criteria in one view}
The later sections give complete proofs for nonlinear diffusion.
Write $p=m>1$ and $d=q-p+1$.
For elliptic nonnegative weak supersolutions in
$W^{1,p}_{\mathrm{loc}}(M)\cap L^q_{\mathrm{loc}}(M)$, Theorem~\ref{thm:pE} proves
\begin{equation}\label{eq:introNonlinearE}
 q>p-1,\qquad
 \int^\infty
 \frac{R^{pq/(p-1)-1}}{V_o(R)^{d/(p-1)}}\dd R=\infty
 \quad\Longrightarrow\quad u=0.
\end{equation}
For parabolic nonnegative distributional supersolutions in the local spatial-energy class specified in Section~\ref{sec:mP}, Theorem~\ref{thm:mP} proves
\begin{equation}\label{eq:introNonlinearP}
 q>\max\{1,p-1\},\qquad
 \int^\infty\frac{R^{p-1}}{V_o(R)^d}\dd R=\infty
 \quad\Longrightarrow\quad u=0.
\end{equation}
In particular, \eqref{eq:introNonlinearP} excludes nonzero global Cauchy solutions in that class with a nonzero nonnegative local $L^1$ initial trace.
The two implications reduce to the linear-diffusion criteria when $p=2$.
Both integral criteria hold without any restriction on the sectional or Ricci curvature of $M$; their proofs use only first-order distance cutoffs and require no Laplacian comparison theorem.
No time derivative of the weak solution is postulated: forward Steklov averages justify the nonlinear temporal test.

The elliptic estimate retains a transition mass to the power $(p-1)/q$.
The parabolic estimate has two different mass powers, $(p-1)/q$ and $1/q$.
Their coexistence calls for a two-exponent mass lemma and the intrinsic time length
\[
 T(R)=R^pV_o(2R)^{p-2}.
\]
When $1<p<2$, this time length need not be monotone. A record-scale selection preserves the divergent series while producing nested cylinders.
At the critical volume degrees, the resulting single-logarithmic exponents are
$(p-1)/(q-p+1)$ in the elliptic case and $1/(q-p+1)$ in the parabolic case.
All fixed finite iterated-logarithmic endpoints are included.
The paper also treats $q=p-1$ elliptically through a spectral criterion and gives counterexamples to a formal continuation of the integral formulas beyond their reaction ranges.

\subsection{The common mechanism and further consequences}
For linear diffusion, the central estimate retains the nonlinear mass $T$ on the region where the cutoff changes, rather than replacing it prematurely by the full localized mass $I$.
Two tests with a fixed negative exponent produce
\[
 I\le C G^{(q-1)/q}T^{1/q}.
\]
The same estimate gives both $I\le KG$ and $T\ge cG^{-(q-1)}I^q$.
On geometric scales these become a mass upper bound and a mass increment lower bound.
An elementary finite-difference estimate for a negative power of the cumulative mass then forces a reciprocal-volume series to converge whenever the supersolution is nonzero.
An index-shifted comparison converts that series into the required integral using only volume monotonicity, with no doubling assumption and no preliminary global $L^q$ bound.

For the heat equation with $p=2$, a nonincreasing temporal cutoff produces a favorable initial boundary contribution in each test.
All spatial and temporal derivative terms are localized to one spacetime transition region.
The time scale $R^2$ adds a factor $R^2$ to the volume, accounting for the change from the elliptic numerator $R^{2q-1}$ to the parabolic numerator $R$.
This gives a self-contained first-order proof of the two linear-diffusion implications in the stated classes; the nonlinear arguments are provided in Sections~7 and~8.
The comparison above records the scope of the cited versions and is not an exhaustive claim of publication priority.

Section~2 contains the summation and monotone-volume lemmas.
Section~3 proves the elliptic theorem, including the legality of bounded Sobolev tests.
Section~4 treats the time boundary terms and proves the parabolic theorem.
Section~5 gives center independence and finite iterated-logarithmic endpoints.
Section~6 provides smooth elliptic models and supercritical Euclidean existence examples.
Sections~7 and~8 prove the nonlinear elliptic and parabolic criteria, respectively.

\Needspace{12\baselineskip}
\section{Discrete mass estimates and monotone volume comparison}
\begin{lemma}[Summation of negative powers]\label{lem:seq}
Let $\gamma,K,c>0$, and let positive sequences $F_k,g_k$ satisfy
\begin{equation}\label{eq:seq}
 F_k\le Kg_k,\qquad F_{k+1}-F_k\ge cg_k^{-\gamma}F_k^{1+\gamma}.
\end{equation}
Then
\begin{equation}\label{eq:seqsum}
 \sum_{k=0}^\infty g_k^{-\gamma}
 \le\frac{(1+cK^\gamma)^{\gamma+1}}{\gamma c}F_0^{-\gamma}.
\end{equation}
Neither monotonicity of $g_k$ nor an upper bound on the actual relative increments of $F_k$ is required.
\end{lemma}
\begin{proof}
Set $z_k=c(F_k/g_k)^\gamma$. Then $0<z_k\le M_0:=cK^\gamma$, and
$F_{k+1}\ge F_k(1+z_k)$. Since a negative power is decreasing and
\[
 1-(1+z)^{-\gamma}=\gamma\int_0^z(1+y)^{-\gamma-1}\dd y
 \ge\frac{\gamma z}{(1+M_0)^{\gamma+1}}\quad(0\le z\le M_0),
\]
we obtain
\[
 F_k^{-\gamma}-F_{k+1}^{-\gamma}
 \ge\frac{\gamma c}{(1+cK^\gamma)^{\gamma+1}}g_k^{-\gamma}.
\]
Sum from $k=0$ to $N$, use $F_{N+1}^{-\gamma}\ge0$, and let $N\to\infty$.
Here the controlled quantity $z_k$ is a lower bound for the actual relative increment; no upper bound for that increment is used.
\end{proof}

\begin{remark}
The mass upper bound cannot be omitted. Indeed, take $g_k=1$,
$F_{k+1}=F_k+F_k^{1+\gamma}$, and $F_0>0$. Every $F_k$ with finite index is finite, and the increment condition holds, whereas $\sum g_k^{-\gamma}=\infty$.
Conversely, $g_k=F_k$ and $F_{k+1}\ge2F_k$ satisfy the hypotheses with $K=c=1$, even when the actual relative increments are arbitrarily large.
\end{remark}

\begin{lemma}[Integral comparison without volume doubling]\label{lem:int}
Let $V:(0,\infty)\to(0,\infty)$ be nondecreasing, let $A>0$, $\gamma>0$, and $L>1$, and set $x_k=L^kx_0$ and $b_k=x_k^A/V(x_k)^\gamma$.
Then
\begin{equation}\label{eq:intcompare}
 \frac{1-L^{-A}}A b_{k+1}
 \le\int_{x_k}^{x_{k+1}}\frac{R^{A-1}}{V(R)^\gamma}\dd R
 \le\frac{L^A-1}A b_k.
\end{equation}
Consequently, the corresponding tail integral converges if and only if $\sum_k b_k$ converges.
\end{lemma}
\begin{proof}
On $[x_k,x_{k+1}]$, bound the denominator using
$V(x_k)\le V(R)\le V(x_{k+1})$ and integrate $R^{A-1}$ exactly.
The lower estimate involves $b_{k+1}$, whereas the upper estimate involves $b_k$.
Summation over finite ranges followed by passage to the limit proves the equivalence, since removing one finite initial term does not affect convergence.
\end{proof}

\section{Proof of the elliptic criterion}
\subsection{Admissibility of bounded Sobolev test functions}
\begin{lemma}\label{lem:test}
Suppose that $u\in W^{1,2}_{\rm loc}(M)$ and $f\in L^1_{\rm loc}(M)$, and that
$\int\langle\nabla u,\nabla\zeta\rangle\ge\int f\zeta$ holds for every nonnegative smooth compactly supported test function.
Then the same inequality holds for every nonnegative bounded compactly supported $W^{1,2}$ test function.
\end{lemma}
\begin{proof}
Let $K$ be the compact support of the test function, and choose
$K\Subset\Omega\Subset M$. Cover $K$ by finitely many relatively compact coordinate neighborhoods, and choose nonnegative smooth functions $\eta_i$ supported in these neighborhoods such that $\sum_i\eta_i=1$ in a neighborhood of $K$.
Extend each $\eta_i\zeta$ by zero in coordinates and convolve with a nonnegative standard mollifier. Choose the mollification radii small enough that all supports remain in one compact set $K'\Subset\Omega$.
Pulling back and summing gives nonnegative functions $\zeta_j\in C_c^\infty(\Omega)$ such that
\[
 \zeta_j\to\zeta\text{ in }W^{1,2},\qquad
 \zeta_j\to\zeta\text{ a.e. along a subsequence},\qquad
 0\le\zeta_j\le\Big(\sum_i\|\eta_i\|_\infty\Big)\|\zeta\|_\infty.
\]
On compact coordinate sets, the metric coefficients and volume density have uniform bounds, so coordinate Sobolev convergence implies the stated convergence on $M$.
The left-hand side passes to the limit by Cauchy--Schwarz, and the right-hand side by dominated convergence with a constant multiple of $|f|\ind_{K'}$ as the dominating function.
The approximation constant serves only to justify admissibility and does not enter the subsequent energy estimates.
\end{proof}

\subsection{Fixed parameters and two tests}
Fix
\begin{equation}\label{eq:parameters}
 0<\theta<\min\{1,q-1\},\qquad
 p_+=\frac{2q}{\gamma+\theta},\quad p_-=\frac{2q}{\gamma-\theta},\qquad s>p_-.
\end{equation}
For example, one may take $\theta=\tfrac12\min\{1,\gamma\}$ and $s=4q/(\gamma-\theta)$.
Thus $0<1\pm\theta<q$ and $2<p_+<\kappa<p_-<s$, and all these parameters depend only on $q$.
For a compactly supported Lipschitz function $0\le\phi\le1$, write
\[
 S=\{|\nabla\phi|>0\},\quad I=\int u^q\phi^s\dd\mu,\quad
 T=\int_Su^q\phi^s\dd\mu,\quad u_\eps=u+\eps\ (0<\eps\le1),
\]
\[
 A_\eps=\int u_\eps^{-1-\theta}\phi^s|\nabla u|^2\dd\mu,\qquad
 T_\eps=\int_Su_\eps^q\phi^s\dd\mu.
\]
All these quantities are finite. For fixed $\eps$, the function
$z\mapsto(z+\eps)^{-\theta}$ is bounded on $[0,\infty)$ with bounded derivative and admits a bounded Lipschitz extension.
The chain and product rules therefore make $u_\eps^{-\theta}\phi^s$ an admissible test function in Lemma~\ref{lem:test}.
Moreover, $\nabla u=0$ almost everywhere on the level set $\{u=0\}$, so the zero set causes no additional term.

The first test gives
\[
 \theta A_\eps+\int u^q u_\eps^{-\theta}\phi^s\dd\mu
 \le s\int_Su_\eps^{-\theta}\phi^{s-1}\langle\nabla u,\nabla\phi\rangle\dd\mu.
\]
Apply Young's inequality to the right-hand side, absorb $\theta A_\eps/2$, and discard the nonnegative nonlinear term to obtain
\begin{equation}\label{eq:Eenergy}
 A_\eps\le\frac{s^2}{\theta^2}B_\eps,\qquad
 B_\eps=\int_Su_\eps^{1-\theta}\phi^{s-2}|\nabla\phi|^2\dd\mu.
\end{equation}
The conjugate exponents $q/(1-\theta)$ and $q/(\gamma+\theta)$ satisfy
\[
 \Big(s-2-\frac{s(1-\theta)}q\Big)\frac q{\gamma+\theta}=s-p_+>0.
\]
H\"older's inequality and $0\le\phi\le1$ consequently yield
\begin{equation}\label{eq:EB}
 B_\eps\le T_\eps^{(1-\theta)/q}
 \left(\int|\nabla\phi|^{p_+}\dd\mu\right)^{2/p_+}.
\end{equation}
For the second test, use $\phi^s$ and apply Cauchy--Schwarz to get
\begin{equation}\label{eq:Esecond}
 I\le sA_\eps^{1/2}D_\eps^{1/2},\qquad
 D_\eps=\int_Su_\eps^{1+\theta}\phi^{s-2}|\nabla\phi|^2\dd\mu.
\end{equation}
Similarly, the conjugate exponents $q/(1+\theta)$ and $q/(\gamma-\theta)$, together with the remaining cutoff power $s-p_->0$, give
\[
 D_\eps\le T_\eps^{(1+\theta)/q}
 \left(\int|\nabla\phi|^{p_-}\dd\mu\right)^{2/p_-}.
\]
Combining these estimates gives
\[
 I\le\frac{s^2}{\theta}T_\eps^{1/q}
 \|\nabla\phi\|_{p_+}\|\nabla\phi\|_{p_-}.
\]
Since $(u+\eps)^q\le2^{q-1}(u^q+1)$ is integrably dominated on the compact support, we may let $\eps\downarrow0$ in the nonsingular mass $T_\eps$.
We have proved
\begin{equation}\label{eq:Ekey}
 \boxed{I\le C_E T^{1/q}\|\nabla\phi\|_{p_+}\|\nabla\phi\|_{p_-},
 \qquad C_E=s^2/\theta.}
\end{equation}
No limit of $A_\eps$ is needed. Throughout the argument the right-hand side of the original differential inequality remains $u^q$; no supersolution inequality with right-hand side $(u+\eps)^q$ is assumed.

\subsection{Radial cutoffs and mass increments}
Define
\[
 \chi(z)=\begin{cases}1,&0\le z\le1,\\2-z,&1<z<2,\\0,&z\ge2,\end{cases}
 \qquad \phi_R=\chi(r/R),\qquad W_o(R)=\mu(\{R<r<2R\}).
\]
Then $\supp\phi_R\subset\overline{B_o(2R)}$ and
$|\nabla\phi_R|\le R^{-1}\ind_{\{R<r<2R\}}$ almost everywhere.
The cut locus does not affect this Lipschitz weak-gradient estimate.
Noncompactness and connectedness ensure $0<W_o(R)<\infty$.
Using
\[
 \frac1{p_+}+\frac1{p_-}=\frac\gamma q,\qquad \frac{\kappa\gamma}q=2,
\]
we obtain
\begin{equation}\label{eq:Eradial}
 I_R\le C_E R^{-2}W_o(R)^{\gamma/q}T_R^{1/q}
 \le C_E\left(\frac{V_o(2R)}{R^\kappa}\right)^{\gamma/q}T_R^{1/q}.
\end{equation}
Fix $G(R)=V_o(2R)/R^\kappa$; alternatively, apply the argument separately with
$G(R)=W_o(R)/R^\kappa$.
For either choice, the same estimate with the same $G$ yields
\begin{equation}\label{eq:Edual}
 I_R\le K G(R),\qquad T_R\ge cG(R)^{-\gamma}I_R^q,
 \qquad K=C_E^{q/\gamma},\quad c=C_E^{-q}.
\end{equation}
The first inequality follows from $T_R\le I_R$; both inequalities hold directly when $I_R=0$.
In particular, $cK^\gamma=1$.

If $u\not\equiv0$, exhaustion by balls and local integrability give $R_*>0$ such that
$F_0=\int_{B_o(R_*)}u^q\dd\mu>0$.
Set
\[
 R_k=4^kR_*,\qquad F_k=\int_{B_o(R_k)}u^q\dd\mu,\qquad g_k=G(R_k).
\]
The transition region $\{R_k<r<2R_k\}$ lies in
$B_o(R_{k+1})\setminus B_o(R_k)$, and the cutoff weight is at most one. Hence
\[
 F_k\le I_{R_k},\qquad T_{R_k}\le F_{k+1}-F_k.
\]
Thus \eqref{eq:Edual} supplies both hypotheses of Lemma~\ref{lem:seq}, which gives
\begin{equation}\label{eq:Eseries}
 \sum_{k=0}^\infty\frac{R_k^{2q}}{V_o(2R_k)^\gamma}
 \le\frac{2^{\gamma+1}C_E^q}{\gamma}F_0^{-\gamma}.
\end{equation}
Apply Lemma~\ref{lem:int} with $A=2q$, $L=4$, and $x_k=2R_k$, so that
$b_k=2^{2q}g_k^{-\gamma}$. We obtain the quantitative necessary estimate
\begin{equation}\label{eq:Equant}
 \int_{2R_*}^\infty\frac{R^{2q-1}}{V_o(R)^\gamma}\dd R
 \le\frac{4^{2q}-1}{2q}\,2^{2q}
 \frac{2^{\gamma+1}C_E^q}{\gamma}F_0^{-\gamma}<\infty.
\end{equation}
This contradicts \eqref{eq:Ecrit} and proves Theorem~\ref{thm:E}.
The annular choice of $G$ gives the following refinement.

\begin{corollary}[Elliptic annular criterion]\label{cor:ann}
If, for some $R_*>0$ and $R_k=4^kR_*$,
\[
 \sum_{k=0}^\infty\frac{R_k^{2q}}{W_o(R_k)^{q-1}}=\infty,
\]
then every supersolution in \eqref{eq:Eweak} vanishes almost everywhere.
\end{corollary}
If the initial ball has zero mass, start at the first ball with positive mass; deleting finitely many terms does not change divergence.
Since $W_o$ need not be monotone, the annular criterion is retained in series form.

\section{The heat equation: temporal boundaries and spacetime mass}
\subsection{Admissible tests after a positive time}
Let $u$ satisfy the classical supersolution hypotheses of Theorem~\ref{thm:P}, and fix $a>0$.
On $M\times[a,\infty)$, take a Lipschitz cutoff $0\le\phi\le1$ with compact spatial support that vanishes for sufficiently large times, and require $\phi_t\le0$.
All tests are first integrated over a compact spatial support and a finite time interval.
Classical regularity, local boundedness, and fixed $\eps>0$ justify all integrals and the temporal chain rule; spatial Lipschitz cutoffs are justified by local Sobolev approximation.
Choose $\theta,s$ as in \eqref{eq:parameters}. All spacetime integrals below use the measure $\dd\mu\dd t$. Write
\[
 S=\{|\nabla\phi|+|\phi_t|>0\},\quad
 I=\iint u^q\phi^s,\quad T=\iint_Su^q\phi^s,\quad
 T_\eps=\iint_Su_\eps^q\phi^s,
\]
\[
 A_\eps=\iint u_\eps^{-1-\theta}\phi^s|\nabla u|^2,\quad
 B_\eps=\iint_Su_\eps^{1-\theta}\phi^{s-2}|\nabla\phi|^2,
\]
\[
 P_\eps=\iint_Su_\eps^{1-\theta}\phi^{s-1}(-\phi_t),\qquad
 D_\eps=\iint_Su_\eps^{1+\theta}\phi^{s-2}|\nabla\phi|^2.
\]
First test with $u_\eps^{-\theta}\phi^s$.
Integration by parts in time gives the exact identity
\begin{align}\label{eq:timeboundary}
 \iint u_tu_\eps^{-\theta}\phi^s
 &=-\frac1{1-\theta}\int_Mu_\eps(x,a)^{1-\theta}\phi(x,a)^s\dd\mu
   +\frac{s}{1-\theta}P_\eps.
\end{align}
The terminal contribution vanishes with the cutoff, and the initial contribution is nonpositive. This sign permits the parabolic extension.
Combining with the spatial cross term, applying Young's inequality, and discarding the nonnegative term
$\iint u^qu_\eps^{-\theta}\phi^s$, we obtain
\begin{equation}\label{eq:Penergy}
 A_\eps\le \frac{s^2}{\theta^2}B_\eps
       +\frac{2s}{\theta(1-\theta)}P_\eps.
\end{equation}
The original differential inequality still has right-hand side $u^q$.

For the second test, use $\phi^s$. Integration by parts in time yields
\[
 I+\int_Mu(x,a)\phi(x,a)^s\dd\mu
 \le s\iint u\phi^{s-1}(-\phi_t)
      +s\iint\phi^{s-1}\langle\nabla u,\nabla\phi\rangle.
\]
Consequently,
\begin{equation}\label{eq:Psecond}
 I\le s A_\eps^{1/2}D_\eps^{1/2}+s Z,\qquad
 Z=\iint_Su\phi^{s-1}(-\phi_t).
\end{equation}
The initial mass is nonnegative on the left and may therefore be discarded without an error term.

\subsection{H\"older factors and a common transition mass}
Define the finite cutoff energies
\[
 E_\pm=\iint|\nabla\phi|^{p_\pm},\quad
 J_+=\iint(-\phi_t)^{q/(\gamma+\theta)},\quad
 J_0=\iint(-\phi_t)^{q/\gamma}.
\]
The two spatial applications of H\"older's inequality are identical to those in the elliptic case and give
\[
 B_\eps\le T_\eps^{(1-\theta)/q}E_+^{2/p_+},\qquad
 D_\eps\le T_\eps^{(1+\theta)/q}E_-^{2/p_-}.
\]
For the temporal energy term, use the conjugate exponents
$q/(1-\theta)$ and $q/(\gamma+\theta)$. The remaining cutoff power is
\[
 \Big(s-1-\frac{s(1-\theta)}q\Big)\frac q{\gamma+\theta}
 =s-\frac q{\gamma+\theta}>0,
\]
so $P_\eps\le T_\eps^{(1-\theta)/q}J_+^{(\gamma+\theta)/q}$.
Similarly, $s>p_->q/\gamma$ ensures
\[
 Z\le T^{1/q}J_0^{\gamma/q}\le T_\eps^{1/q}J_0^{\gamma/q}.
\]
Substituting into \eqref{eq:Penergy}--\eqref{eq:Psecond}, and setting
$A_1=s^2/\theta^2$ and $A_2=2s/[\theta(1-\theta)]$, we find
\begin{equation}\label{eq:Pkey}
 I\le sT_\eps^{1/q}
 \left[(A_1E_+^{2/p_+}+A_2J_+^{(\gamma+\theta)/q})^{1/2}
 E_-^{1/p_-}+J_0^{\gamma/q}\right].
\end{equation}
The exponents of the two energy masses add to $1/q$, and every mass is retained on the same set $S$.
For a fixed cutoff, local integrable domination allows $\eps\downarrow0$, replacing $T_\eps$ by $T$.

\subsection{Parabolic scaling and the integral criterion}
Take
\[
 \phi_R(x,t)=\chi(r(x)/R)\chi((t-a)/R^2),\qquad
 \Q_R^a=B_o(R)\times(a,a+R^2),
\]
where $\chi$ is the piecewise linear profile of the preceding section.
Then $\phi_R=1$ on $\Q_R^a$, its support is contained in
$\overline{B_o(2R)}\times[a,a+2R^2]$, and
\[
 |\nabla\phi_R|\le R^{-1},\quad -\partial_t\phi_R\le R^{-2},\quad
 S_R\subset(B_o(2R)\times(a,a+2R^2))\setminus\Q_R^a
\]
up to end faces and spheres, which do not affect the integrals.
Let $H_R=2R^2V_o(2R)$. Then
\[
 E_\pm\le R^{-p_\pm}H_R,\quad
 J_+\le R^{-2q/(\gamma+\theta)}H_R,\quad
 J_0\le R^{-2q/\gamma}H_R.
\]
In \eqref{eq:Pkey}, both the squared spatial term and the temporal term have the scale
$R^{-2}H_R^{(\gamma+\theta)/q}$.
It follows that
\begin{equation}\label{eq:Pradial}
 \boxed{I_R\le C_P R^{-2}H_R^{\gamma/q}T_R^{1/q}
 =C_P G_P(R)^{\gamma/q}T_R^{1/q},\qquad
 G_P(R)=\frac{2R^2V_o(2R)}{R^\kappa}.}
\end{equation}
For example, $C_P=s(\sqrt{A_1+A_2}+1)$ is admissible and depends only on $q$.
As in the elliptic case, we obtain
\[
 I_R\le K_PG_P(R),\qquad
 T_R\ge c_PG_P(R)^{-\gamma}I_R^q,
 \quad K_P=C_P^{q/\gamma},\quad c_P=C_P^{-q}.
\]
For fixed $a>0$, if $u$ is nonzero at some time $t>a$, choose $R_*$ large enough that
\[
 F_0=\iint_{\Q_{R_*}^a}u^q>0.
\]
This uses only exhaustion by cylinders and continuity.
Set
\[
 R_k=4^kR_*,\qquad F_k=\iint_{\Q_{R_k}^a}u^q,\qquad g_k=G_P(R_k).
\]
Since $a+2R_k^2<a+R_{k+1}^2$ and $2R_k<R_{k+1}$, the entire transition region lies in
$\Q_{R_{k+1}}^a\setminus\Q_{R_k}^a$.
Thus again $F_k\le I_{R_k}$ and $T_{R_k}\le F_{k+1}-F_k$.
Lemma~\ref{lem:seq} gives
\begin{equation}\label{eq:Pseries}
 \sum_k g_k^{-\gamma}<\infty,\qquad
 g_k^{-\gamma}=2^{-\gamma}\frac{R_k^2}{V_o(2R_k)^\gamma},
\end{equation}
because $\kappa\gamma-2\gamma=2$.
Apply Lemma~\ref{lem:int} with $A=2$, $L=4$, and $x_k=2R_k$ to obtain
\begin{equation}\label{eq:Pquant}
 \int_{2R_*}^\infty\frac{R}{V_o(R)^\gamma}\dd R
 \le \frac{4^2-1}{2}\,2^{\gamma+2}
 \frac{2^{\gamma+1}C_P^q}{\gamma}F_0^{-\gamma}<\infty.
\end{equation}
This contradicts \eqref{eq:Pcrit}.
Hence $u$ vanishes after every $a>0$.
Any nonzero point at a positive time would contradict this conclusion upon choosing a smaller $a$, so $u\equiv0$ on the open cylinder.
A Cauchy solution with continuous initial trace then has $u_0=0$, proving Theorem~\ref{thm:P}.

\begin{remark}
The spatial annular volume $W_o(R)$ alone cannot replace the entire spacetime transition volume in \eqref{eq:Pradial}: when the temporal cutoff varies, points in the inner spatial ball also belong to the transition region.
Keeping the actual spacetime volume of $S_R$ yields a corresponding stronger series criterion; here we use only the upper bound $2R^2V_o(2R)$.
\end{remark}

\section{Center independence and endpoint consequences}
\subsection{Lower limits and centers}
For $0<A<B<\infty$, the bound $V_o(R)\ge V_o(A)>0$ makes both criterion integrands integrable on $[A,B]$.
Thus changing a positive finite lower limit does not affect divergence.
If $d(o,z)=d_0$, then, for sufficiently large $R$,
\[
 B_z(R/2)\subset B_o(R)\subset B_z(2R).
\]
Use these inclusions in the denominators and change variables by $R/2$ and $2R$.
For every $A>0$, it follows that divergence of
$\int^\infty R^{A-1}/V_o(R)^\gamma\dd R$ is independent of the center.
No uniform comparison of volumes with different centers at the same radius is needed.

\subsection{Finitely iterated logarithms}
Let $L_1(R)=\log R$ and $L_{j+1}(R)=\log L_j(R)$, fix a finite integer $N\ge1$, and restrict to $R$ large enough that all logarithms exceed one.
The critical spatial volume powers for the two criteria are, respectively,
\[
 d_E=\frac{2q}{q-1},\qquad d_P=\frac{2}{q-1}.
\]
\begin{corollary}\label{cor:logs}
Choose either the elliptic or the parabolic setting and put $d=d_E$ or $d=d_P$, respectively.
If, for all sufficiently large $R$,
\begin{equation}\label{eq:logvol}
 V_o(R)\le CR^d\prod_{j=1}^{N-1}L_j(R)^\beta L_N(R)^b,
 \qquad b\le\beta,
\end{equation}
then the vanishing conclusion of the corresponding theorem holds.
In particular, all $N$ logarithmic powers may simultaneously equal $\beta$.
\end{corollary}
\begin{proof}
In both settings, the volume upper bound gives the following lower bound for the criterion integrand:
\[
 \frac{C^{-\gamma}}{R\prod_{j<N}L_j(R)\,L_N(R)^{b\gamma}}.
\]
Using $L_N'(R)=[R\prod_{j<N}L_j(R)]^{-1}$ and setting $z=L_N(R)$, the comparison integral becomes
$C^{-\gamma}\int^\infty z^{-b\gamma}\dd z$, which diverges when $b\gamma\le1$.
The single, double, and triple logarithmic cases correspond to $N=1,2,3$.
At the simultaneous endpoint, an antiderivative is $L_{N+1}(R)$.
\end{proof}
For a general finite product $\prod_{j=1}^NL_j^{b_j}$, the model integral also diverges if the first exponent differing from $\beta$ is smaller than $\beta$, and converges if that exponent is larger than $\beta$.
Indeed, remove the preceding critical factors and set $z=L_j(R)$.
The remaining finite product of slowly varying factors, denoted by $h(z)$, satisfies
$z^{-\eta}\le h(z)\le z^\eta$ for every $\eta>0$ and all sufficiently large $z$, with constants absorbed by increasing the lower threshold.
Choose $\eta<|b_j\gamma-1|$ and compare with power integrals.
This assertion concerns fixed finite logarithmic depth only.
When $b>\beta$, the upper bound \eqref{eq:logvol} alone does not imply convergence of the criterion integral for the actual volume function.

\section{Existence examples and the limits of the criteria}
\subsection{A Euclidean supercritical elliptic example}
Suppose that $n>2$ and $q>n/(n-2)$, and set
$\alpha=2/(q-1)\in(0,n-2)$ and $f(x)=(1+|x|^2)^{-\alpha/2}$.
Direct differentiation gives
\[
 \begin{gathered}
 -\Delta f=\alpha(1+r^2)^{-\alpha/2-2}
 [n+(n-\alpha-2)r^2]>0,\\
 \frac{f^q}{-\Delta f}=\frac{1+r^2}{\alpha[n+(n-\alpha-2)r^2]}
 \le\frac1{\alpha(n-\alpha-2)}.
 \end{gathered}
\]
Therefore, for a sufficiently small constant $\lambda>0$, the function
$u=\lambda f$ satisfies $-\Delta u>u^q$ on all of $\R^n$.
It is smooth at the origin and has no singularity.

\subsection{Smooth elliptic models with finitely iterated logarithms}
We next give a global flux construction showing that the final logarithmic exponent in the elliptic part of Corollary~\ref{cor:logs} cannot, in general, be increased.
Fix a topological dimension $n\ge2$, a finite $N$, and $b>\beta$.
Choose $0<\delta<b-\beta$ and write
\[
 \ell(r)=\prod_{j<N}L_j(r)^\beta L_N(r)^b,\qquad P(r)=L_N(r)^\delta.
\]
Define these functions beyond a sufficiently large $R_1$.
Choose a smooth transition $0\le\eta\le1$, flat at its endpoints, equal to zero on $[0,R_1]$ and to one on $[2R_1,\infty)$.
Define the density
\[
 A(r)=r^{n-1}\ (r\le R_1),\qquad
 A(r)=\exp\{(1-\eta)(n-1)\log r+
 \eta[(\kappa-1)\log r+\log\ell(r)]\}\ (r>R_1).
\]
Then $A>0$ for $r>0$, and $A=r^{\kappa-1}\ell$ in the tail.
The metric $g=\dd r^2+A(r)^{2/(n-1)}g_{S^{n-1}}$ agrees with the Euclidean metric near the origin and extends smoothly to $\R^n$.
The radial coordinate is the distance from the pole, and closed radial balls are compact. Thus the manifold is complete, and
\[
 V_o(R)=\omega_{n-1}\int_0^R A(r)\dd r
 \sim\frac{\omega_{n-1}}\kappa R^\kappa\ell(R).
\]
This asymptotic relation follows directly from $r\ell'(r)/\ell(r)\to0$ and integral comparison.

To ensure monotonicity throughout the transition, first prescribe a strictly positive flux derivative:
\[
 h(r)=nr^{n-1}\ (r\le R_1),\qquad
 h(r)=(1-\eta)nr^{n-1}+\eta P'(r)\ (r>R_1),\qquad
 H(r)=\int_0^rh(s)\dd s.
\]
Then $H'>0$ for $r>0$, $H=r^n$ near the origin, and
$H=P+C_0\sim P$ in the tail.
Set
\[
 U(r)=\int_r^\infty\frac{H(s)}{A(s)}\dd s.
\]
Since $\kappa>2$, the tail is integrable. Near the origin $H/A=r$, so $U(0)<\infty$ and
\[
 U(r)=U(0)-r^2/2\quad(r\text{ near }0),\qquad
 -\Delta U=H'/A>0.
\]
In particular, $U$ is smooth at the pole and $-\Delta U=n$ there; no point source is present.
In the tail,
\[
 U(r)\sim\frac{r^{2-\kappa}}{\kappa-2}\prod_{j<N}L_j(r)^{-\beta}L_N(r)^{\delta-b},
\]
\[
 -\Delta U=\delta r^{-\kappa}\prod_{j<N}L_j(r)^{-\beta-1}L_N(r)^{\delta-b-1}.
\]
Using $(\kappa-2)q=\kappa$ and $\beta q=\beta+1$, we find
\[
 \frac{U^q}{-\Delta U}\sim
 \frac{L_N(r)^{1-(b-\delta)\gamma}}{\delta(\kappa-2)^q}\longrightarrow0.
\]
The ratio is continuous and bounded on the whole manifold.
Thus, for sufficiently small $\lambda>0$, the function $\lambda U$ is a global strictly positive supersolution.
This proves existence on the specified models; it does not assert that convergence of the volume integral suffices for existence on an arbitrary manifold.

\subsection{Supercritical Euclidean heat flow with small initial data}
To verify the sharpness of the power in the parabolic criterion, assume
$q>1+2/n$, and put $d=n(q-1)/2>1$.
The function
\[
 h(x,t)=(1+t)^{-n/2}\exp\{-|x|^2/[4(1+t)]\}
\]
satisfies $h_t-\Delta h=0$ and $h^{q-1}\le(1+t)^{-d}$.
Choose $a_0>0$ such that $a_0^{-\gamma}>\gamma/(d-1)$, and let
\[
 a(t)=\left[a_0^{-\gamma}-\gamma\int_0^t(1+s)^{-d}\dd s\right]^{-1/\gamma}.
\]
Then $a$ is globally bounded and $a'=a^q(1+t)^{-d}$.
Consequently, $U=a(t)h(x,t)$ satisfies
\[
 U_t-\Delta U=a^q(1+t)^{-d}h\ge a^qh^q=U^q.
\]
For any nonzero $0\le u_0\in C_b(\R^n)$ with
$u_0\le a_0e^{-|x|^2/4}$, start with $v_0(t)=e^{t\Delta}u_0$ and iterate
\[
 v_{j+1}(t)=e^{t\Delta}u_0+\int_0^te^{(t-s)\Delta}v_j(s)^q\dd s.
\]
Positivity of the Euclidean heat semigroup and the supersolution $U$ give
$0\le v_j\le v_{j+1}\le U$.
Dominated convergence on each finite time interval yields a global mild solution.
Standard local regularization by the Euclidean heat kernel then gives a classical solution at positive times with continuous initial trace.
Thus nonzero small initial data do admit global Cauchy solutions when $q>1+2/n$.
This Euclidean construction does not provide a convergence-side existence theorem on arbitrary manifolds, nor does it establish model sharpness at every iterated logarithmic endpoint for the heat equation.

\section{Quasilinear elliptic inequalities: the \texorpdfstring{$p$}{p}-Laplacian}
\label{sec:pelliptic}
The auxiliary linear-diffusion volume exponent was denoted by $\kappa$. We now write
\[
 \Delta_p u=\operatorname{div}(|\nabla u|^{p-2}\nabla u),\qquad p>1.
\]
Thus $\Delta_m$ is $\Delta_p$ with $p=m$; $p$ always denotes the diffusion exponent in this section.
We interpret $|\nabla u|^{p-2}\nabla u$ as zero where $\nabla u=0$.
The elliptic extension has a complete proof with the same two-test mechanism, although the mass exponent changes.

\begin{theorem}[Integral criterion for the $p$-Laplacian]\label{thm:pE}
Let $p>1$, $q>p-1$, and set
\[
 \delta=q-p+1,\qquad \Gamma=\frac{\delta}{p-1},\qquad
 D=\frac{pq}{\delta}.
\]
If
\begin{equation}\label{eq:pEcrit}
 \int_{R_0}^{\infty}
 \frac{R^{pq/(p-1)-1}}{V_o(R)^{(q-p+1)/(p-1)}}\dd R=\infty,
\end{equation}
then every nonnegative $u\in W^{1,p}_{\mathrm{loc}}(M)\cap L^q_{\mathrm{loc}}(M)$ satisfying
\begin{equation}\label{eq:pweak}
 \int_M |\nabla u|^{p-2}\langle\nabla u,\nabla\zeta\rangle\dd\mu
 \ge\int_M u^q\zeta\dd\mu
 \quad(0\le\zeta\in C_c^\infty(M))
\end{equation}
vanishes almost everywhere.
No curvature assumption is imposed on $M$.
\end{theorem}

\subsection{The two tests and their admissibility}
The bounded-test approximation used previously extends from $W^{1,2}$ to $W^{1,p}$.
Indeed, the flux in \eqref{eq:pweak} belongs locally to $L^{p/(p-1)}$, so strong $W^{1,p}$ approximation passes the left-hand side to the limit.
Nonnegative bounded approximation, with fixed compact support, gives dominated convergence on the right-hand side because $u^q\in L^1_{\mathrm{loc}}$.
This permits every nonnegative bounded compactly supported $W^{1,p}$ test.

Choose fixed parameters
\begin{equation}\label{eq:pparameters}
 0<\vartheta<\min\left\{p-1,\frac{\delta}{p-1}\right\},\qquad
 a_+=\frac{pq}{\delta+\vartheta},\quad
 a_-=\frac{pq}{\delta-(p-1)\vartheta},\qquad s>a_-.
\end{equation}
They satisfy $p<a_+<D<a_-<s$.
Given $0\le\phi\le1$ compactly supported and Lipschitz, write
\[
 S=\{|\nabla\phi|>0\},\qquad
 I=\int u^q\phi^s,\qquad T=\int_Su^q\phi^s,\qquad
 v=u+\eps,
\]
and define
\begin{align*}
 A_\eps&=\int v^{-1-\vartheta}\phi^s|\nabla u|^p,\\
 B_\eps&=\int_S v^{p-1-\vartheta}\phi^{s-p}|\nabla\phi|^p,\\
 D_\eps&=\int_S v^{(p-1)(1+\vartheta)}\phi^{s-p}|\nabla\phi|^p,\qquad
 T_\eps=\int_Sv^q\phi^s.
\end{align*}
Here and below an unspecified elliptic integral uses Riemannian volume.
For each fixed $\eps>0$, $v^{-\vartheta}\phi^s$ is a bounded admissible Sobolev test, by the chain rule.
Testing with this function and retaining the sign of the reaction term gives
\[
 \vartheta A_\eps+\int u^qv^{-\vartheta}\phi^s
 \le s\int_Sv^{-\vartheta}\phi^{s-1}
              |\nabla u|^{p-1}|\nabla\phi|
 \le s A_\eps^{(p-1)/p}B_\eps^{1/p}.
\]
Consequently,
\begin{equation}\label{eq:penergy}
 A_\eps\le(s/\vartheta)^pB_\eps.
\end{equation}
The pure cutoff test $\phi^s$ gives
\begin{equation}\label{eq:psecond}
 I\le s A_\eps^{(p-1)/p}D_\eps^{1/p}.
\end{equation}
The exponents $p-1-\vartheta$ and $(p-1)(1+\vartheta)$ both lie strictly between zero and $q$.
H\"older's inequality therefore yields
\begin{align}\label{eq:pHolder}
 B_\eps&\le T_\eps^{(p-1-\vartheta)/q}
                  \|\nabla\phi\|_{a_+}^{p},&
 D_\eps&\le T_\eps^{(p-1)(1+\vartheta)/q}
                  \|\nabla\phi\|_{a_-}^{p}.
\end{align}
For example, the remaining cutoff power in the first application, after raising it to the conjugate exponent, is $s-a_+>0$; the second one is $s-a_->0$.
The two mass powers combine exactly as
\[
 \frac{(p-1)(p-1-\vartheta)+(p-1)(1+\vartheta)}{pq}
 =\frac{p-1}{q}.
\]
The estimate $v^q\le\max\{1,2^{q-1}\}(u^q+1)$ supplies local domination even if $q<1$. Combining \eqref{eq:penergy}--\eqref{eq:pHolder} and letting $\eps\downarrow0$ only in $T_\eps$ proves
\begin{equation}\label{eq:pkey}
 \boxed{\ I\le C_p T^{(p-1)/q}
          \|\nabla\phi\|_{a_+}^{p-1}\|\nabla\phi\|_{a_-},
       \qquad C_p=\frac{s^p}{\vartheta^{p-1}}.\ }
\end{equation}
There is no need to pass a singular weighted gradient integral to the limit.

\subsection{Radial localization and the volume integral}
Use $\phi_R=\chi(r/R)$ and $W_o(R)=\mu(\{R<r<2R\})$.
Since
\[
 \frac{p-1}{a_+}+\frac1{a_-}=\frac{\delta}{q},
\]
\eqref{eq:pkey} implies
\begin{equation}\label{eq:pradial}
 I_R\le C_pR^{-p}W_o(R)^{\delta/q}T_R^{(p-1)/q}
 \le C_pG(R)^{\delta/q}T_R^{(p-1)/q},\qquad
 G(R)=\frac{V_o(2R)}{R^D}.
\end{equation}
Using $T_R\le I_R$ and then rearranging the very same estimate gives
\begin{equation}\label{eq:pdual}
 I_R\le K_pG(R),\qquad
 T_R\ge c_pG(R)^{-\Gamma}I_R^{1+\Gamma},\qquad
 K_p=C_p^{q/\delta},\quad c_p=C_p^{-q/(p-1)}.
\end{equation}
In particular $c_pK_p^{\Gamma}=1$.
If $u\not\equiv0$, choose $R_*$ with $F_0=\int_{B_o(R_*)}u^q>0$, and let
\[
 R_k=4^kR_*,\qquad F_k=\int_{B_o(R_k)}u^q,\qquad g_k=G(R_k).
\]
The inclusion of transition regions and \eqref{eq:pdual} imply
\[
 F_k\le K_pg_k,\qquad
 F_{k+1}-F_k\ge c_pg_k^{-\Gamma}F_k^{1+\Gamma}.
\]
Lemma~\ref{lem:seq}, now with exponent $\Gamma$, gives the quantitative necessary condition
\begin{equation}\label{eq:pseries}
 \sum_{k=0}^\infty
 \frac{R_k^{pq/(p-1)}}{V_o(2R_k)^{\Gamma}}
 \le\frac{2^{\Gamma+1}}{\Gamma c_p}F_0^{-\Gamma}<\infty.
\end{equation}
The monotone-volume comparison in Lemma~\ref{lem:int}, with power $pq/(p-1)$ and geometric ratio $4$, converts this to convergence of the integral in \eqref{eq:pEcrit}.
This contradiction proves Theorem~\ref{thm:pE}.
Replacing $V_o(2R)$ by $W_o(R)$ throughout gives the further sufficient condition
\[
 \sum_{k=0}^\infty
 \frac{R_k^{pq/(p-1)}}{W_o(R_k)^{\Gamma}}=\infty.
\]
This refinement is kept in series form, since annular volumes need not be monotone.

\subsection{Logarithmic endpoints and the Euclidean threshold}
At the critical polynomial volume exponent $D=pq/(q-p+1)$, the estimate
\[
 V_o(R)\le CR^D(\log R)^b,\qquad
 b\le\frac{p-1}{q-p+1},
\]
implies \eqref{eq:pEcrit}.
More generally, for every fixed positive integer $N$, the endpoint
\begin{equation}\label{eq:plogs}
 V_o(R)\le CR^D
 \left(\prod_{j=1}^{N}\log_jR\right)^{(p-1)/(q-p+1)}
\end{equation}
also implies nonexistence, where $\log_1R=\log R$ and
$\log_{j+1}R=\log(\log_jR)$ for sufficiently large $R$.
Indeed, the integrand is then bounded below by a constant multiple of
$[R\prod_{j=1}^{N}\log_jR]^{-1}$, whose integral diverges.

For $M=\R^n$, \eqref{eq:pEcrit} is equivalent to $n\le D$.
Thus for $n>p$ it yields the full range
\begin{equation}\label{eq:pEuclidean}
 p-1<q\le\frac{n(p-1)}{n-p},
\end{equation}
including equality, while for $n\le p$ it excludes every $q>p-1$.
The Euclidean upper endpoint is sharp even for weak supersolutions.
To verify this directly, assume $n>p$ and $q>n(p-1)/(n-p)$, and set
\[
 b=\frac{p-1}{q-p+1},\qquad \tau=\frac p{p-1},\qquad
 U(x)=A(1+|x|^\tau)^{-b}.
\]
Here $c_0=n-p(b+1)>0$, and radial differentiation gives
\begin{align*}
 -\Delta_pU
  &=(Ab\tau)^{p-1}(1+r^\tau)^{-(b+1)(p-1)-1}
                     [n+c_0r^\tau]\\
  &\ge (b\tau)^{p-1}c_0A^{-(q-p+1)}U^q.
\end{align*}
For sufficiently small $A>0$, this is a positive supersolution of
$-\Delta_pU\ge U^q$ on all of $\R^n$.
Its radial flux has zero flux through shrinking spheres at the origin, so the calculation proves the global weak inequality.
The local $W^{1,p}$ regularity is sufficient; $C^2$ regularity at the origin is not claimed when $p>2$.

\subsection{The critical reaction and a spectral alternative}
Define the bottom of the $p$-spectrum by
\[
 \lambda_{1,p}(M)=
 \inf_{0\ne\psi\in C_c^\infty(M)}
 \frac{\int_M|\nabla\psi|^p\dd\mu}{\int_M|\psi|^p\dd\mu}.
\]
At $q=p-1$, a different sufficient condition is available.
\begin{proposition}[A critical-reaction criterion]\label{prop:pcritical}
If $\lambda_{1,p}(M)<1$, there is no nonzero nonnegative weak supersolution of
$-\Delta_pu\ge u^{p-1}$ in $W^{1,p}_{\mathrm{loc}}(M)$.
In particular this holds if
\begin{equation}\label{eq:pentropy}
 h_o:=\limsup_{R\to\infty}\frac{\log V_o(R)}R<p.
\end{equation}
\end{proposition}
\begin{proof}
A nonzero nonnegative weak $p$-superharmonic function on a connected manifold is positive almost everywhere, by the local strong minimum principle.
This standard fact can also be obtained from the logarithmic energy test
$(u+\eps)^{1-p}\phi^p$: it bounds
$\int\phi^p|\nabla\log(u+\eps)|^p$ uniformly in $\eps$.
On a ball with a zero set of positive measure, the local Poincar\'e inequality for functions vanishing on that set then bounds
$\log(1+u/\eps)$ in $L^p$, forcing $u=0$ on the ball; overlapping balls and connectedness propagate this conclusion.

For $0\le\phi\in C_c^\infty(M)$, test the supersolution inequality with
$\phi^p/(u+\eps)^{p-1}$.
Young's inequality in the form
$pa^{p-1}b-(p-1)a^p\le b^p$ gives
\[
 \int_M\left(\frac{u}{u+\eps}\right)^{p-1}\phi^p\dd\mu
 \le\int_M|\nabla\phi|^p\dd\mu.
\]
Since $u>0$ almost everywhere, letting $\eps\downarrow0$ shows
$\lambda_{1,p}(M)\ge1$.
The same Rayleigh infimum may be taken over nonnegative compactly supported Lipschitz functions, by approximation.

To verify the last assertion, fix $\alpha>h_o/p$ when $h_o<\infty$.
The definition of $h_o$ gives $e^{-\alpha r}\in W^{1,p}(M)$ and
\[
 \int_M|\nabla e^{-\alpha r}|^p\dd\mu
 \le\alpha^p\int_Me^{-p\alpha r}\dd\mu<\infty.
\]
Multiplying by $\chi(r/R)$ and sending $R\to\infty$ produces admissible Rayleigh tests and proves
$\lambda_{1,p}(M)\le\alpha^p$.
Letting $\alpha\downarrow h_o/p$ yields
$\lambda_{1,p}(M)\le(h_o/p)^p$.
Thus \eqref{eq:pentropy} implies $\lambda_{1,p}(M)<1$.
\end{proof}

\subsection{Counterexamples outside the reaction range}
The restriction $q>p-1$ enters the admissible parameters, the absorption of the transition mass, and the positive summation exponent $\Gamma$.
At $q=p-1$ the problem has a spectral character.
For example, on $n$-dimensional hyperbolic space with $n\ge2$ and constant curvature $-\kappa_H^2$, let $b$ be a Busemann function with $|\nabla b|=1$ and $\Delta b=h=(n-1)\kappa_H$.
Then
\[
 u=e^{-ab},\qquad
 -\Delta_pu=a^{p-1}[h-(p-1)a]u^{p-1}.
\]
Choosing $a=h/p$ gives the coefficient $(h/p)^p$.
When $h\ge p$, this is a positive supersolution with $q=p-1$; when $h=p$, it is an exact solution.
In particular the integral formula cannot be continued to $q=p-1$ by setting its volume exponent equal to zero.

The absence of a general theorem below that reaction threshold can also be seen without a spectral existence theorem.
Fix $n\ge2$ and an even integer $a>p$, and consider the complete rotational model
\[
 g=\dd r^2+f(r)^2g_{\mathbb S^{n-1}},\qquad f(r)=r e^{r^a}.
\]
Its area and volume functions are
$A(r)=|\mathbb S^{n-1}|f(r)^{n-1}$ and $V(r)=\int_0^rA(s)\dd s$.
The positive function
\[
 \mathcal T(r)=\int_r^\infty
          \left(\frac{V(s)}{A(s)}\right)^{1/(p-1)}\dd s
\]
is bounded: near the pole $V/A\sim r/n$, and at infinity
$V/A\sim[(n-1)a r^{a-1}]^{-1}$, which is integrable after raising to $1/(p-1)$ because $a>p$.
The flux identity
\[
 A(r)|\mathcal T'(r)|^{p-2}\mathcal T'(r)=-V(r)
\]
shows that $-\Delta_p\mathcal T=1$ weakly on the entire model, including the pole.
Normalize $U=\mathcal T/\mathcal T(0)$, so $0<U\le1$.
Under the complete rescaled metric $g_L=L^2g$,
\[
 -\Delta_{p,g_L}U=L^{-p}\mathcal T(0)^{-(p-1)}.
\]
Taking $L$ sufficiently small gives $-\Delta_{p,g_L}U\ge1\ge U^q$ for every $q>0$.
Consequently, a nonexistence assertion valid on every complete manifold cannot be obtained for $0<q\le p-1$ by formally continuing \eqref{eq:pEcrit}.
This example is consistent with Theorem~\ref{thm:pE}, since its superexponential volume makes the integral in \eqref{eq:pEcrit} converge whenever $q>p-1$.

\section{Intrinsic cylinders for the parabolic \texorpdfstring{$m$}{m}-Laplacian}\label{sec:mP}
Write $p=m>1$ in this section, and set
\[
 \Delta_pu=\operatorname{div}(|\nabla u|^{p-2}\nabla u),\qquad
 q>\max\{1,p-1\},\qquad d=q-p+1>0.
\]
The time scale requires a new argument when $p\ne2$: multiplying a solution by a constant changes the relative strength of its time derivative and its diffusion. The fixed heat scale $R^2$ therefore cannot simply be retained.

\subsection{Solution class and the integral criterion}
For precision, we use the following distributional energy class:
\[
 u\in L^q_{\mathrm{loc}}(M\times(0,\infty))
 \cap L^p_{\mathrm{loc}}((0,\infty);W^{1,p}_{\mathrm{loc}}(M)).
\]
The nonnegative function $u$ is a supersolution when
\begin{equation}\label{eq:mPweak}
 \iint \bigl(-u\,\partial_t\zeta+
 |\nabla u|^{p-2}\langle\nabla u,\nabla\zeta\rangle\bigr)
 \dd\mu\dd t\ge\iint u^q\zeta\dd\mu\dd t
 \quad(0\le\zeta\in C_c^\infty(M\times(0,\infty))).
\end{equation}
No local boundedness or integrable time derivative is assumed. The temporal regularization needed for the nonlinear test is proved below. Smooth supersolutions for which the diffusion is interpreted weakly are included; the flux need not be differentiable at a point where $\nabla u=0$.

\begin{theorem}[Intrinsic parabolic integral criterion]\label{thm:mP}
Let $p>1$ and $q>\max\{1,p-1\}$. If
\begin{equation}\label{eq:mPcrit}
 \int_{R_0}^{\infty}\frac{R^{p-1}}{V_o(R)^{q-p+1}}\dd R=\infty,
\end{equation}
then every nonnegative global supersolution of \eqref{eq:mPweak} in the class just specified vanishes almost everywhere on $M\times(0,\infty)$.
Consequently, the Cauchy problem
\[
 u_t-\Delta_pu=u^q,\qquad u(\cdot,0)=u_0\ge0,
\]
has no nonzero global solution in this class with a local $L^1$ initial trace $u_0\not\equiv0$.
No curvature, doubling, Poincar\'e, or heat-kernel hypothesis is imposed.
\end{theorem}

\subsection{A two-exponent mass lemma}
\begin{lemma}\label{lem:mPmass}
Let $0<a,c<1$, let $(h_k)$ be any positive sequence, and let $(F_k)$ be a positive nondecreasing sequence. Suppose, for a fixed $C>0$, that
\begin{align}
 F_k&\le C\max\{h_k^{-1/\gamma_1},h_k^{-1/\gamma_2}\},
 \label{eq:mPupper}\\
 F_k&\le C\left[
 \left(\frac{F_{k+1}-F_k}{h_k}\right)^a+
 \left(\frac{F_{k+1}-F_k}{h_k}\right)^c\right],
 \label{eq:mPincrement}
\end{align}
where $\gamma_1=a^{-1}-1$ and $\gamma_2=c^{-1}-1$. Then $\sum_kh_k<\infty$.
\end{lemma}
\begin{proof}
Put $\gamma=\min(\gamma_1,\gamma_2)$ and $\Gamma=\max(\gamma_1,\gamma_2)$.
At least one term in \eqref{eq:mPincrement} is at least $F_k/(2C)$, and hence
\[
 F_{k+1}-F_k\ge c_0h_k
 \min\{F_k^{1+\gamma},F_k^{1+\Gamma}\}
 \ge c_1h_kF_k^{1+\gamma}.
\]
The last constant may depend on the positive number $F_0$, because $F_k\ge F_0$.
The upper bound also gives $h_kF_k^\gamma\le K$ uniformly: for $h_k\le1$, this follows directly from \eqref{eq:mPupper}; for $h_k>1$, the same bound and $F_k\ge F_0$ first give $h_k\le(C/F_0)^\Gamma$, after enlarging $C$ if necessary. More explicitly, in the latter case
\[
 h_kF_k^\gamma\le C^\gamma h_k^{1-\gamma/\Gamma}
 \le C^\gamma\max\{1,(C/F_0)^{\Gamma-\gamma}\}.
\]
Apply Lemma~\ref{lem:seq} with $g_k=h_k^{-1/\gamma}$ and the corresponding enlarged constants. It yields $\sum_kh_k<\infty$.
\end{proof}

\subsection{Two tests on one cylinder}
Fix a positive starting time $a_0$, a spatial radius $R$, and a time length $T$.
Choose a product cutoff $0\le\phi\le1$ equal to one on
$B_o(R)\times[a_0,a_0+T]$, supported in
$\overline{B_o(2R)}\times[a_0,a_0+2T]$, nonincreasing in time, with
\[
 |\nabla\phi|\le C/R,\qquad |\phi_t|\le C/T.
\]
The cutoff is understood on the relative half-cylinder $t\ge a_0$; no extension before $a_0$ is required. All time integrals below start at $a_0$. The initial boundary contributions have a favorable sign and are retained or discarded accordingly.
Choose a fixed number
\[
 0<\theta<\min\{1,p-1,q/(p-1)-1\},
\]
and take a fixed cutoff power satisfying
\[
 s>\max\left\{
 \frac{pq}{d+\theta},\frac{pq}{d-(p-1)\theta},
 \frac{q}{q-1+\theta},\frac{q}{q-1}\right\}.
\]
All these parameters can be fixed in terms of $p,q$ alone.
We first justify the nonlinear test without assuming a time derivative of $u$. Write $\mathcal A=|\nabla u|^{p-2}\nabla u$ and take forward Steklov averages
\[
 u_h(t)=h^{-1}\int_t^{t+h}u(s)\dd s,
 \quad \mathcal A_h(t)=h^{-1}\int_t^{t+h}\mathcal A(s)\dd s,
 \quad f_h(t)=h^{-1}\int_t^{t+h}u(s)^q\dd s.
\]
They satisfy $\partial_tu_h-\operatorname{div}\mathcal A_h\ge f_h\ge0$ on a slightly smaller cylinder, with $\partial_tu_h=(u(t+h)-u(t))/h\in L^1_{\mathrm{loc}}$.
The averaged inequality permits nonnegative bounded spatial Sobolev tests: bounded approximation with almost-everywhere convergence passes the terms involving $\partial_tu_h$ and $f_h$ by dominated convergence, while strong convergence of spatial gradients passes the flux term by H\"older's inequality. This approximation may be performed in spacetime with a uniform bound and a fixed compact support; after averaging, no time derivative of the test appears in this extension. Time-interval indicators are allowed by the same integrable approximation.

For fixed $\varepsilon>0$, use $(u_h+\varepsilon)^{-\theta}\phi^s$ as test, integrate from $a_0$, and discard the nonnegative averaged reaction term. The temporal primitive is
\[
 P_\varepsilon(v)=
 \frac{(v+\varepsilon)^{1-\theta}-\varepsilon^{1-\theta}}{1-\theta}.
\]
Its derivative is bounded on $[0,\infty)$, so the temporal $W^{1,1}$ chain rule applies to $u_h$. The initial contribution $-\int P_\varepsilon(u_h(a_0))\phi(a_0)^s$ moves to the favorable side and may be discarded. The gradient pairing $\langle\mathcal A_h,\nabla u_h\rangle$ is not assumed nonnegative when $h>0$.

First let $h\downarrow0$ while keeping $\varepsilon$ fixed. Locally, $\nabla u_h\to\nabla u$ strongly in $L^p$ and $\mathcal A_h\to\mathcal A$ strongly in $L^{p/(p-1)}$, so
$\langle\mathcal A_h,\nabla u_h\rangle\to|\nabla u|^p$ strongly in $L^1$.
The regularized weights are uniformly bounded and converge almost everywhere along a subsequence; their weighted products therefore converge in $L^1$. Also $P_\varepsilon(u_h)\to P_\varepsilon(u)$ in $L^1$ because this primitive is Lipschitz.
Only after this passage apply Young's inequality, and finally let $\varepsilon\downarrow0$. This gives
\begin{equation}\label{eq:mPenergy}
 A:=\iint u^{-1-\theta}|\nabla u|^p\phi^s
 \le C(B+H),
\end{equation}
where
\[
 B=\iint u^{p-1-\theta}|\nabla\phi|^p\phi^{s-p},
 \qquad
 H=\iint u^{1-\theta}|\phi_t|\phi^{s-1}.
\]
Here and below zero is assigned to the energy integrand on $\{u=0\}$; the spatial Sobolev gradient is zero almost everywhere on this set.
In the last passage use Fatou's lemma for $A$ and dominated convergence for $B,H$; the powers $p-1-\theta$ and $1-\theta$ are positive and less than $q$.

The pure cutoff test in the averaged inequality gives a favorable initial term $\int u_h(a_0)\phi(a_0)^s$. Discard it and let $h\downarrow0$, using $f_h\to u^q$ and $u_h\to u$ in $L^1$ and $\mathcal A_h\to\mathcal A$ in $L^{p/(p-1)}$. Thus, with
$I=\iint u^q\phi^s$,
\begin{equation}\label{eq:mPpure}
 I\le C A^{(p-1)/p}D^{1/p}+CZ,
 \quad
 D=\iint u^{(1+\theta)(p-1)}|\nabla\phi|^p\phi^{s-p},
 \quad
 Z=\iint u|\phi_t|\phi^{s-1}.
\end{equation}
Again the initial contribution is nonnegative on the left.
Let $W=2TV_o(2R)$, and let $J$ denote either the weighted mass $I$ or the unweighted $u^q$ mass on a set containing the supports of both cutoff derivatives.
H\"older's inequality bounds $B,H,D,Z$ by
\begin{align*}
 B&\le CR^{-p}W^{1-(p-1-\theta)/q}J^{(p-1-\theta)/q},\\
 H&\le CT^{-1}W^{1-(1-\theta)/q}J^{(1-\theta)/q},\\
 D&\le CR^{-p}W^{1-(1+\theta)(p-1)/q}J^{(1+\theta)(p-1)/q},\\
 Z&\le CT^{-1}W^{1-1/q}J^{1/q}.
\end{align*}
For the choice $J=I$, retain the factors $\phi^s$ in H\"older's inequality. For example, the remaining power of $\phi$ in the $B$ estimate is nonnegative once
$s\ge p/[1-(p-1-\theta)/q]$. The analogous requirements for $H,D,Z$ are all finite by the choice of $\theta$.
For the unweighted transition mass, simply discard the remaining nonnegative cutoff powers. Thus the two choices of $J$ are justified separately; an unweighted transition mass is not assumed bounded by $I$.

Set
\[
 \alpha=\frac{p-1}{q},\quad \chi=\frac1q,
 \qquad
 A_R=R^{-p}W^{1-\alpha},\quad A_T=T^{-1}W^{1-\chi}.
\]
Substitution in \eqref{eq:mPpure} gives the pure terms
$A_RJ^\alpha$ and $A_TJ^\chi$, and the mixed term
\[
 (A_RJ^\alpha)^{1/p}(A_TJ^\chi)^{(p-1)/p}.
\]
The last identity follows from
$2(p-1)/(pq)=\alpha/p+(p-1)\chi/p$.
Weighted arithmetic--geometric mean therefore gives
\begin{equation}\label{eq:mPlocal}
 I\le C(A_RJ^\alpha+A_TJ^\chi).
\end{equation}

Choose the volume-dependent intrinsic time and its associated weight as
\begin{equation}\label{eq:mPtime}
 T(R)=R^pV_o(2R)^{p-2},
 \qquad h(R)=\frac{R^p}{V_o(2R)^d}.
\end{equation}
With $W=2T(R)V_o(2R)$, direct cancellation of the powers gives
$A_R\le Ch(R)^{-\alpha}$ and $A_T\le Ch(R)^{-\chi}$.
Consequently \eqref{eq:mPlocal} becomes
\begin{equation}\label{eq:mPmaster}
 I\le C\left[(J/h(R))^\alpha+(J/h(R))^\chi\right].
\end{equation}
Choosing $J=I$ first yields
\begin{equation}\label{eq:mPcutoffupper}
 I\le C\max\{h(R)^{-(p-1)/d},h(R)^{-1/(q-1)}\}.
\end{equation}

\subsection{Nesting intrinsic time scales without volume doubling}
For $R_k=2^kR_0$ write $T_k=T(R_k)$ and $h_k=h(R_k)$.
In Lemma~\ref{lem:int}, take $A=p$, $L=2$ and $x_k=2R_k$. Its comparison terms are $b_k=2^ph_k$, so \eqref{eq:mPcrit} is equivalent to $\sum_kh_k=\infty$. No comparison between $V_o(R_k)$ and $V_o(2R_k)$ is made.
If $p\ge2$, volume monotonicity gives $T_{k+1}\ge2^pT_k$, so the associated cylinders are already nested with room for the cutoff.
For $1<p<2$ the time lengths need not be monotone. This difficulty can be resolved by selecting scales rather than assuming a geometric regularity condition.

\begin{lemma}[Selection of increasing intrinsic times]\label{lem:mPrecords}
Suppose $1<p<2$ and $q>1$. If $\sum_kh_k=\infty$, there are indices
$j_0<j_1<\cdots$ such that $T_{j_{\ell+1}}\ge2T_{j_\ell}$ and
$\sum_\ell h_{j_\ell}=\infty$.
\end{lemma}
\begin{proof}
Put $\rho=d/(2-p)>1$ and $\lambda=p(\rho-1)>0$.
The definitions imply the exact identity
\[
 h_k=R_k^{-\lambda}T_k^\rho.
\]
If $T_k$ were bounded, the $h_k$ series would converge geometrically. Thus $T_k$ is unbounded.
Starting from any $j_0$, choose $j_{\ell+1}$ to be the first index larger than $j_\ell$ at which $T_k\ge2T_{j_\ell}$.
For $j_\ell\le k<j_{\ell+1}$,
\[
 h_k\le2^\rho\,2^{-\lambda(k-j_\ell)}h_{j_\ell}.
\]
Summation over each block gives
$\sum_kh_k\le C\sum_\ell h_{j_\ell}$, proving the claim.
\end{proof}

\begin{proof}[Proof of Theorem~\ref{thm:mP}]
Suppose that $u$ is nonzero. Choose $a_0>0$ and a sufficiently large initial cylinder so its $u^q$ mass is positive. For $p\ge2$, use every dyadic index; for $1<p<2$, use Lemma~\ref{lem:mPrecords} and start sufficiently far out. Denote the selected radii, times, and weights by $\widehat R_\ell,\widehat T_\ell,\widehat h_\ell$.
Then $\widehat R_{\ell+1}\ge2\widehat R_\ell$,
$\widehat T_{\ell+1}\ge2\widehat T_\ell$, and
$\sum_\ell\widehat h_\ell=\infty$.
Define
\[
 Q_\ell=B_o(\widehat R_\ell)\times(a_0,a_0+\widehat T_\ell),
 \qquad F_\ell=\iint_{Q_\ell}u^q>0.
\]
The cutoff at scale $\ell$ is one on $Q_\ell$, and all its derivatives are supported in $Q_{\ell+1}\setminus Q_\ell$.
Equation~\eqref{eq:mPcutoffupper} gives the upper mass estimate in Lemma~\ref{lem:mPmass}, with
$a=\alpha$, $c=\chi$, $\gamma_1=d/(p-1)$, and $\gamma_2=q-1$.
Equation~\eqref{eq:mPmaster}, now using the unweighted transition mass, gives its increment estimate because
$I\ge F_\ell$ and $J\le F_{\ell+1}-F_\ell$.
That lemma forces $\sum_\ell\widehat h_\ell<\infty$, a contradiction.
The local $L^1$ initial trace of an identically zero spacetime solution is zero, proving the Cauchy assertion.
\end{proof}

\subsection{Thresholds, logarithms, and limitations}
If, for large $R$,
\[
 V_o(R)\le CR^{p/d}(\log R)^b,
 \qquad b\le\frac1d,
\]
then \eqref{eq:mPcrit} holds. More generally the finite iterated-logarithmic bounds
\[
 V_o(R)\le CR^{p/d}
 \prod_{j=1}^{N}(\log_jR)^{1/d}
\]
also imply nonexistence, since the reciprocal-volume integrand dominates
$[R\prod_{j=1}^N\log_jR]^{-1}$.
On $\mathbb R^n$ the criterion is exactly
\[
 q\le p-1+\frac pn,
 \qquad q>\max\{1,p-1\}.
\]
This recovers the nonexistence side of the Fujita threshold within the stated regularity class, including its endpoint whenever that endpoint is greater than one. If $p\le2n/(n+1)$, this threshold is at most one, so the criterion supplies no superlinear Euclidean nonexistence interval.

For $0<q\le1$ there is a direct obstruction to an analogous global Liouville conclusion: a spatially constant solution satisfies $y'=y^q$. For $y(0)>0$ it is global, with
$y(t)=y(0)e^t$ if $q=1$ and
$y(t)=[y(0)^{1-q}+(1-q)t]^{1/(1-q)}$ if $0<q<1$.
For $1<q\le p-1$, the stationary supersolution furnished by the rapidly expanding elliptic model in Section~\ref{sec:pelliptic} shows that a vanishing conclusion in that range needs an additional geometric restriction. This example addresses supersolutions of the inequality only.
Finally, convergence of \eqref{eq:mPcrit} is not an existence theorem on an arbitrary manifold.


\begin{thebibliography}{99}
\footnotesize
\setlength{\itemsep}{1pt}
\bibitem{MP01}
E.~Mitidieri and S.~I.~Pohozaev,
A priori estimates and the absence of solutions of nonlinear partial differential equations and inequalities,
\emph{Proc. Steklov Inst. Math.} \textbf{234} (2001), no.~3, 1--362.

\bibitem{GK10}
A.~Grigor'yan and V.~A.~Kondratiev,
On the existence of positive solutions of semi-linear elliptic inequalities on Riemannian manifolds,
in \emph{Around the Research of Vladimir Maz'ya II: Partial Differential Equations},
A.~Laptev (ed.), International Mathematical Series, vol.~12,
Springer, New York, 2010, pp.~203--218.

\bibitem{GS14}
A.~Grigor'yan and Y.~Sun,
On nonnegative solutions of the inequality $\Delta u+u^\sigma\le0$ on Riemannian manifolds,
\emph{Comm. Pure Appl. Math.} \textbf{67} (2014), 1336--1352.

\bibitem{Z15}
H.-C.~Zhang,
A note on Liouville type theorem of elliptic inequality $\Delta u+u^\sigma\le0$ on Riemannian manifolds,
\emph{Potential Anal.} \textbf{43} (2015), 269--276.

\bibitem{Xu19}
F.~Xu,
On the existence of nonnegative solutions to semilinear differential inequality on Riemannian manifolds,
\emph{Nonlinear Anal.} \textbf{183} (2019), 29--41.

\bibitem{GS19}
A.~Grigor'yan and Y.~Sun,
On positive solutions of semi-linear elliptic inequalities on Riemannian manifolds,
\emph{Calc. Var. Partial Differential Equations} \textbf{58} (2019), article~207.

\bibitem{GSV20}
A.~Grigor'yan, Y.~Sun and I.~Verbitsky,
Superlinear elliptic inequalities on manifolds,
\emph{J. Funct. Anal.} \textbf{278} (2020), article~108444.

\bibitem{Fujita}
H.~Fujita,
On the blowing up of solutions of the Cauchy problem for $u_t=\Delta u+u^{1+\alpha}$,
\emph{J. Fac. Sci. Univ. Tokyo Sect. I} \textbf{13} (1966), 109--124.

\bibitem{Hayakawa1973}
K.~Hayakawa,
On nonexistence of global solutions of some semilinear parabolic differential equations,
\emph{Proc. Japan Acad.} \textbf{49} (1973), 503--505.

\bibitem{Sugitani1975}
S.~Sugitani,
On nonexistence of global solutions for some nonlinear integral equations,
\emph{Osaka J. Math.} \textbf{12} (1975), 45--51.

\bibitem{KST1977}
K.~Kobayashi, T.~Sirao and H.~Tanaka,
On the growing up problem for semilinear heat equations,
\emph{J. Math. Soc. Japan} \textbf{29} (1977), 407--424.

\bibitem{MMP2017}
P.~Mastrolia, D.~D.~Monticelli and F.~Punzo,
Nonexistence of solutions to parabolic differential inequalities with a potential on Riemannian manifolds,
\emph{Math. Ann.} \textbf{367} (2017), 929--963.

\bibitem{Xu2018}
F.~Xu,
On existence and nonexistence of nonnegative solutions to seminlinear differential equation on Riemannian manifolds,
preprint, arXiv:1811.12731v1 (2018).

\bibitem{SX2022}
Y.~Sun and F.~Xu,
On existence and nonexistence of nonnegative solutions to heat equation on Riemannian manifolds,
in \emph{Geometric Potential Analysis}, M.~Milman, J.~Xiao and B.~Zegarlinski (eds.),
Advances in Analysis and Geometry, vol.~6,
De Gruyter, Berlin/Boston, 2022, pp.~207--218.

\bibitem{RY2022}
M.~Ruzhansky and N.~Yessirkegenov,
Existence and non-existence of global solutions for semilinear heat equations and inequalities on sub-Riemannian manifolds, and Fujita exponent on unimodular Lie groups,
\emph{J. Differential Equations} \textbf{308} (2022), 455--473.

\bibitem{CGM2025}
D.-E.~von Criegern, G.~Grillo and D.~D.~Monticelli,
Nonexistence of solutions to classes of parabolic inequalities in the Riemannian setting,
\emph{J. London Math. Soc.} \textbf{112} (2025), e70394.

\bibitem{GHHS2026}
Q.~Gu, L.~Hao, X.~Huang and Y.~Sun,
Flow Decomposition and Sharp Integral Fujita Criteria on Weighted Graphs,
preprint, arXiv:2608.06715 (2026).
\bibitem{MMP15}
P.~Mastrolia, D.~D.~Monticelli and F.~Punzo,
Nonexistence results for elliptic differential inequalities with a potential on Riemannian manifolds,
\emph{Calc. Var. Partial Differential Equations} \textbf{54} (2015), 1345--1372.

\bibitem{Sun16}
Y.~Sun,
Uniqueness result on nonnegative solutions of a large class of differential inequalities on Riemannian manifolds,
\emph{Pacific J. Math.} \textbf{280} (2016), no.~1, 241--254.

\bibitem{GMP23}
G.~Grillo, G.~Meglioli and F.~Punzo,
Global existence for reaction-diffusion evolution equations driven by the $p$-Laplacian on manifolds,
\emph{Math. Eng.} \textbf{5} (2023), no.~3, 1--38.

\bibitem{GHHS26p}
Q.~Gu, L.~Hao, X.~Huang and Y.~Sun,
A Volume-Growth Criterion for the $p$-Laplace Inequality on Weighted Graphs,
preprint, arXiv:2605.10446v2 (2026).

\end{thebibliography}
\end{document}